\documentclass{amsart}

\usepackage{graphicx}
\usepackage{psfrag} 
\usepackage{mathptmx}      
\usepackage{amssymb}
\usepackage{amsmath}
\usepackage{color}

\usepackage{setspace}

\usepackage[mathscr]{euscript}

\newtheorem{theorem}{Theorem}[section]
\newtheorem{lemma}[theorem]{Lemma} 
\newtheorem{corollary}[theorem]{Corollary}
\newtheorem{proposition}[theorem]{Proposition}

\newcommand{\cl}{{\mathcal L}}

\newcommand{\cL}{{\mathcal L}}

\begin{document}

\title{Lassoing and corraling rooted phylogenetic trees
}


\author{Katharina T.\,Huber}


\address{School of Computing Sciences, University of East Anglia,
        Norwich, NR4 7TJ, UK}
\email{katharina.huber@cmp.uea.ac.uk}

\author{Andrei-Alin Popescu }
 \address{School of Computing Sciences, University of East Anglia,
        Norwich, NR4 7TJ, UK}
        \email{Andrei-Alin.Popescu@uea.ac.uk}

\date{\today\phantom{Received: date / Accepted: date}}

\maketitle

\begin{abstract}
The construction of a dendogram on a set of individuals 
is a key component of a genomewide association study. 
However even with modern 
sequencing technologies the distances on the 
individuals required for the construction of such a structure
may not always be reliable 
making it tempting to exclude them from an analysis. This, in turn, 
results in an input set for dendogram construction that consists 
of only partial distance information which raises the
following fundamental question. For what subset of its leaf set can we reconstruct uniquely the dendogram from the distances that it induces on that subset.
By formalizing a dendogram in terms of an
edge-weighted, rooted phylogenetic tree on a pre-given finite set 
$X$ with $|X|\geq 3$ whose edge-weighting is equidistant
and a set of partial
distances on $X$ in terms of a set $\cL$ of 2-subsets of $X$,
we investigate this problem in terms of when such a tree is lassoed, that is,
uniquely determined by the elements in $\cL$. For this we consider four 
different formalizations of the idea of ``uniquely determining'' 
giving rise to four distinct types of lassos. We present 
characterizations for all of them in terms of 
the child-edge graphs of the interior vertices of such a tree.
Our characterizations imply in particular that in case the tree
in question is binary then all four types of lasso must coincide.

\end{abstract}

\noindent{\bf Keywords} {rooted phylogenetic tree, lasso, corral,
partial distance

\section{Introduction}

Years of selective breeding have resulted in large numbers of different
varieties for, for example, oilseed rape and rice and also numerous 
animal breeds including dogs 
and chicken. Genomewide association studies constitute a powerful 
tool to try and
link the observed phenotypic variability between the varieties (we
will collectively refer to a variety and a breed as a variety)
with variations in the genomes of the variety. A key 
component of such a study is phenetic clustering whereby one aims to 
construct a dendogram for a set of individuals from within a variety
of interest indicating levels of similarity between them. Using 
some sort of distance measure this similarity can be based
on e.\,g.\,morphological traits such as grain type or Single Nucleotide 
Polymorphism (SNP) markers obtained through next generation  
sequencing technology (see e.\,g.\, \cite{andrea2012,GWASchicken,mlm} 
for examples of such studies).

From a formal point of view, a dendogram can be thought of as 
a pair $(T,\omega)$ consisting of (i) a rooted  
tree $T$ with leaf set a given non-empty finite set $X$  (e.\,g.\, 
individuals), no degree two vertices 
but possibly a distinguished vertex $\rho_T$ of $T$
called the root of $T$, and all other non-leaf vertices
of $T$ of degree at least three (we will refer to such a
tree simply as an {\em $X$-tree}), and (ii) an edge-weighting
$\omega: E(T)\to \mathbb R_{\geq 0}$ for $T$ that is
equidistant which means that the induced distance 
$D_{(T,\omega)}(\rho_T,x)$ from 
$\rho_T$ to every leaf $x\in X$ of $T$ is the same (as with all relevant
concepts, we refer the
reader to the next section for a precise definition). With and without
the equidistance requirement such pairs $(T,\omega)$ have generated a 
lot of interest in the literature
 and so it is not surprising that numerous deep results for them are 
known  provided the distance
information from which to construct $T$ and $\omega$  is complete
in the sense that for all elements $x$ and $y$ in $X$ the distance
between $x$ and $y$ is given (see e.\,g.\,\cite{basic-phy-comb,SS03}).

However even for data generated with modern 
sequencing technologies the required distance measures 
need not always be reliable (or may simply be missing)
resulting in only partial 
distance information for dendogram reconstruction.
 From the perspective of the aforementioned formalization
of such a structure, the problem thus becomes (a) how to 
construct an {\em equidistant $X$-tree} (i.\,e.\,an $X$-tree 
with equidistant edge weighting) 
from distance information on only a subset of pairs of its
leaves and, (b) if such a tree can be constructed from such a subset of its
leaves, when is it lassoed , that is, uniquely determined by that set.

Although approaches for tackling the first problem 
exist in the form of, for example,
an approach introduced in \cite{D-S84} not much is known about
the second. A notable exception is a study 
in \cite{DHS12} carried out for the unrooted analogue 
of an equidistant $X$-tree.
Viewing partial distance information on a set $X$ as a set of
{\em cords}, that is, subsets of $X$ of size two, 
the authors considered the following four
natural interpretations of the above uniqueness
problem.  Namely, given a set $\cL\subseteq {X\choose 2}$ 
and an edge-weighted unrooted phylogenetic tree $T$ on $X$ when 
\begin{enumerate}
\item[(i)] 
is the edge-weighting of $T$ uniquely determined by $\cL$?,
\item[(ii)]  is the shape i.\,e.\,the topology of $T$ uniquely 
determined by $\cL$?,
\item[(iii)] are both the edge-weighting and the topology of $T$ 
uniquely determined
by $\cL$?, and
\item[(iv)]  when is the topology of $T$ uniquely determined 
by $\cL$ up to $T$ being obtained from another $X$-tree 
by collapsing edges?.
\end{enumerate}

Formalized as $\cL$ being an 
edge-weight/topological/strong/weak lasso for an unrooted
phylogenetic tree with leaf set $X$, 
the authors
of \cite{DHS12} showed that all four concepts are distinct.  
Also, they presented results that allowed them to not only
investigate the above types of lasso 
from a recursive point of view but also characterize
under what circumstances a specifically constructed set of cords
is a topological lasso
(see Section~\ref{sec:examples} for more on this).
However a characterization for the general case eluded them.

Replacing the concept of an edge-weight lasso by that of an 
equidistant lasso 
to reflect the fact that for the edge-weighted 
$X$-trees of interest
here the induced distance from the root to any leaf of such a tree 
is the same, we show that for $X$-trees the situation changes. 
More precisely,
we present for an $X$-tree $T$ characterizations for
when a set $\cL\subseteq {X\choose 2}$ is a weak lasso
for $T$ (Theorem~\ref{theo:characterization-corraling}),
an equidistant lasso for $T$ (Theorem~\ref{theo:edge-weight}),
and for when it is a topological lasso for $T$ 
(Theorem~\ref{theo:characterization-topology})
in terms of the child edge graph $G_T(\cL,v)$ that can be 
canonically associated to every non-leaf vertex $v$ of $T$ 
via its child edges. Our characterizations can be thought 
of as a spectrum on the connectedness of that graph with the
extreme situations being an equidistant lasso and a topological lasso.
They imply that every topological lasso and every non-empty
weak lasso must
be an edge-weight lasso ((Corollaries~\ref{cor:coincide} and 
\ref{cor:weak-implies-equidistant}) and that in case 
$T$ is binary the notions of an equidistant lasso and a topological
lasso (and thus a weak lasso) coincide. Consequently, 
every edge-weight/topological
lasso is also a strong lasso in that case (Corollary~\ref{cor:coincide}). 
We also investigate two
special types of sets of cords originally introduced in \cite{DHS12} in the
light of our findings above. This investigation
shows in particular
that it is possible for the concepts of an equidistant and a topological lasso
to coincide without the $X$-tree they are referring to being binary.

The outline of the paper is as follows.
In the next section, we introduce relevant terminology. In 
Section~\ref{sec:first-characterization}, we present first
characterizations for when a set of cords is a 
topological/equidistant/weak lasso for an $X$-tree 
(Theorem~\ref{theo:first-characterization}). 
In Section~\ref{sec:properties-of-G(cL,v)}, we introduce 
the child-edge graph associated to a non-leaf vertex of an $X$-tree and
present first properties of it concerning corraling 
sets of cords. In Section~\ref{sec:corral}, we establish 
Theorem~\ref{theo:characterization-corraling}.
In Section~\ref{sec:edge-weight}, we show
Theorem~\ref{theo:edge-weight} and in 
Section~\ref{sec:topological},
we prove Theorem~\ref{theo:characterization-topology}. 
In Section~\ref{sec:examples},
we present two general ways for constructing 
for an $X$-tree $T$ two different sets of
cords of ${X\choose 2}$ and discuss their properties in the context of lassoing and corraling $T$. We conclude with
Section~\ref{sec:conclusion} where we also 
present some open problems.

\section{Preliminaries}\label{sec:basic-terminology}

Assume from now on that $X$ is always a finite set with at least 3 
elements. For a cord  of $X$ with elements $a$ and $b$ we write
$ab$ rather than $\{a,b\}$. 

Suppose for the following that  $T$ is an $X$-tree. Then 
we call a vertex of $T$ that is not a leaf of $T$ an {\em interior vertex}
of $T$, an edge $e$ that is incident with a leaf of $T$ a {\em pendant
edge} of $T$, and an edge of $T$ that is incident with two 
interior vertices of $T$ an {\em interior edge} of $T$. 
We denote the set of all interior
vertices of $T$ by $V^o(T)$ and the root of $T$ by
$\rho_T$. Moreover we call $T$ binary if every interior vertex
of $T$ has degree three but the root which has degree two. Let $T'$ be a further $X$-tree. Then we say that 
$T$ and $T'$ are  {\em equivalent} if there exists a
 bijection $\phi:V(T)\to V(T')$
that extends to a graph isomorphism between $T$ and $T'$ that is
the identity on $X$ and maps the root $\rho_T$ of $T$ to the
root $\rho_{T'}$ of $T'$. Also, we say that
$T'$ {\em refines} $T$ if, up to equivalence,
$T$ can be obtained from $T'$ by collapsing edges of $T'$
(see e.\,g.\, \cite{SS03}). In that case, we will also call $T'$
a {\em refinement} of $T$. Note that every $X$-tree is its own refinement.

Let $\omega$ denote an {\em edge-weighting for $T$}, that is, 
a map $\omega:E(T)\to \mathbb R_{\geq 0}$. Then we denote by 
$(T,\omega)$ the edge-weighted $X$-tree whose
underlying $X$-tree and edge-weighting are $T$ and $\omega$, respectively,
and by $D_{(T,\omega)}$ the distance 
induced by $\omega$ on $V(T)$. We call $\omega$ 
{\em proper} if $\omega(e)>0$ holds for all interior edges
$e$ of $T$ and we call it {\em equidistant} if 
\begin{enumerate}
\item[(E1)] $D_{(T,\omega)}(x,\rho_T)=  D_{(T,\omega)}(y,\rho_T)$, for all
$x,y\in X$, and
\item[(E2)] $D_{(T,\omega)}(x,u)\geq  D_{(T,\omega)}(x,v)$, for all
$x\in X$ and any $u,v\in V(T)$ such that $u$ is encountered
before $v$  on the path from $\rho_T$ to $x$.
\end{enumerate} 
Note that Property (E2) implies that, for all interior vertices $v$
of $T$ and all leaves $x,y\in X$ of $T$ for which $v$ lies on the
path from $x$ to $y$, we have 
$D_{(T,\omega)}(x,v)=D_{(T,\omega)}(y,v)$. Also note that
our definition of an equidistant edge-weighting
is slightly different from the one given
in \cite{SS03} in so far that $\omega$ is a map into 
$\mathbb R_{\geq 0}$ and not into $\mathbb R$, as in \cite{SS03}.

Suppose $\cL\subseteq {X\choose 2}$ is a set of cords. 
If $T'$  is a further $X$-tree and
$\omega$ and $\omega'$ are proper edge-weightings for $T$ and
$T'$, respectively, we say 
that $(T,\omega)$ and $(T',\omega')$ are {\em $\cL$-isometric} if 
$D_{(T,\omega)}(x,y)=D_{(T',\omega')}(x,y)$ holds for all cords 
$xy\in \cL$. Canonically extending the corresponding
concepts introduced in \cite{DHS12} 
for unrooted phylogenetic trees on $X$ (and further studied in \cite{tripCovDist}) to $X$-trees, we say that $\cL$ is
\begin{enumerate}
\item[(i)] an {\em equidistant lasso for $T$} if, for all 
equidistant, proper edge-weightings $\omega$ and $\omega'$ of
$T$, we have that $\omega=\omega'$ holds whenever
$(T,\omega)$ and $(T,\omega')$ are $\cL$-isometric
\item[(ii)] a {\em topological lasso for $T$} if, for every $X$-tree $T'$ and
any equidistant, proper edge-weightings $\omega$ of $T$ and $\omega'$ of
$T'$, respectively,  we have that $T$ and $T'$ are equivalent
whenever $(T,\omega)$ and  $(T',\omega')$ are $\cL$-isometric.
\item[(iii)] is a {\em strong lasso for $T$} if $\cL$  is simultaneously
an equidistant and a topological lasso for $T$.
\item[(iv)] a {\em weak lasso for $T$} if, for every $X$-tree $T'$ and
any equidistant, proper edge-weightings $\omega$ of $T$ and $\omega'$ of
$T'$, respectively  we have that $T$ is refined by $T'$ 
whenever $(T,\omega)$ and  $(T',\omega')$ are $\cL$-isometric.
\end{enumerate}
Also, we say that a set $\cL\subseteq {X\choose 2}$ of cords is an {\em 
equidistant/topological/weak/strong lasso on $X$} if there exists an
$X$-tree $T$ such that $\cL$ is an equidistant/topological/weak/strong lasso
for $T$. For the convenience of the reader, we illustrate the above 
types of lassos in Fig.~\ref{examplesFig} for $X=\{a,b,c,d,e\}$.
\begin{figure}[h]
\begin{center}
\includegraphics[scale=0.3]{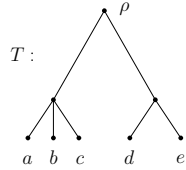}
\caption{\label{examplesFig}
For $X=\{a,b,c,d,e\}$, the set $\cL=\{ab,cd,de\}$ is an equidistant lasso
for the depicted $X$-tree $T$, the set $\cL=\{ab,ac,bc,bd,de\}$ is a topological lasso for $T$
and also a strong lasso for $T$, and the set $\cL=\{ab,bc,cd,de\}$ is a 
weak lasso for $T$.}
\end{center}
\end{figure}

Note that we will also say that a set $\cL$ of cords of $X$ {\em corrals} an
$X$-tree $T$ if $\cL$ is a weak lasso for $T$. Also note that a topological lasso for an $X$-tree $T$ is in particular a weak lasso for $T$, and that the notions of
a topological lasso for $T$ and a weak lasso for $T$ coincide if $T$ 
is binary. Finally note that for $\cL$ to be a topological/equidistant lasso 
we must have that $\cL\not=\emptyset$.
However $\cL\not=\emptyset$ need not hold for $\cL$ to corral
an $X$-tree as every subset of ${X\choose 2}$ including the empty-set 
corrals the {\em star-tree on $X$}, that
is the tree with a unique interior vertex and leaf set $X$. 




\section{A first characterization of a topological/weak/equidistant lasso}
\label{sec:first-characterization}
In this section we present a first characterization for when a set
of cords of $X$ is a topological/weak/equidistant lasso for an
$X$-tree. To establish this characterization, 
we require further definitions and notations. 

Suppose for the following that $T$ is an $X$-tree
and that $v\in V^o(T)$. Then
we call an edge $e$ of $T$ incident
with $v$ that is not crossed by the path from the root
$\rho_T$ of $T$ to $v$ a {\em child edge} of $v$.  If $e$
is incident with $v$ but lies on the path from 
$\rho_T$ to $v$ then we call it a {\em parent edge} of $v$.
In the former case, we call the vertex incident with that
edge but distinct from $v$ a {\em child} of $v$ and in the
latter a {\em parent} of $v$. We call a vertex $w$ of $T$ distinct 
from $v$ a {\em descendant} of $v$ if there exists a path from $v$ to $w$
(possibly of length one) that crosses a child of $v$ and denote the 
set of leaves 
of $T$ that are also descendants of $v$ by $L_T(v)$. If $v$ is a leaf
of $T$, then we put $L_T(v):=\{v\}$. Also if there
is no ambiguity as to which $X$-tree $T$ we are referring then we will
write $L(v)$ rather than $L_T(v)$. 

Suppose $T$ is an $X$-tree. Then, for all $x\in X$, we denote
the edge of $T$ incident with $x$ by $e_x$ and the parent of
$x$ by $v_x$. Moreover, we call a non-empty subset $L\subsetneq X$ of 
leaves of $T$ that all have the same 
parent $v\in V^o(T)$ a {\em pseudo-cherry} of $T$ if $L=L(v)$. In that case,
we also call $v$ the {\em parent} of that pseudo-cherry. 
If $\{x_1,\ldots, x_k\}$, $k\geq 2$, is a pseudo-cherry of
some $X$-tree $T$ then we will sometimes write $x_1,\ldots, x_k$
rather than $\{x_1,\ldots, x_k\}$. Note that every 
$X$-tree on three or more leaves that is {\em non-degenerate}, that is, not the star-tree on $X$ must
contain at least one pseudo-cherry. Also note
that in case $|L|=2$
the definition of an pseudo-cherry reduces to that of  
a cherry in the usual sense (see e.g. \cite{SS03}).
In the special
case that $|X|=3$, say $X=\{a,b,c\}$, and that $T$ has a cherry, 
$a,b$ say,
we call $T$ a {\em triplet} and denote $T$ by $ab|c$ (or,
equivalently, $c|ab$). 

Now assume again that $T$ is an arbitrary $X$-tree. For any 
non-empty subset $Y\subseteq X$, we denote by $T|_Y$  
the subgraph of $T$ whose leaf set is $Y$ (where we {\em suppress} resulting
interior vertices $v$ of $T|_Y$ with a single child, that is, if 
the parent of $v$ is $w$ and $u$ is the unique child of $v$ 
then we delete $v$ plus its incident edges from $T|_Y$ and add the
edge $\{w,u\}$ to $E(T|_Y)$). Note that $T|_Y$ is clearly a $Y$-tree.
We say for three pairwise
distinct elements $a,b,c\in X$ that $T$ {\em displays} the triplet
$ab|c$ if $T|_{\{a,b,c\}}$ is equivalent with
$ab|c$. We denote by $\mathcal R(T)$ the set of triplets displayed by $T$.
As is well-known, any $X$-tree $T$ can display at most ${|X|\choose 3}$
triplets with equality holding if and only if $T$ is binary.
Furthermore, any $X$-tree $T$ is uniquely determined
by the set $\mathcal R(T)$ in the sense that 
if $T'$ is a further $X$-tree and $\mathcal R(T)=\mathcal R(T')$
holds then $T$ and $T'$ must be equivalent (see 
e.\,g.\,\cite[Chapter 9]{basic-phy-comb} and \cite[Section 6.4]{SS03}). 

Observe that if $T$ is an $X$-tree, $\omega$ is an equidistant, 
proper edge-weighting for $T$, and $a,a',b\in X$ are
pairwise distinct elements then either 
all three pairwise distances induced on $Z:=\{a,a',b\}$ 
must coincide or two of the distances induced by it must coincide
and the third one must be strictly less than that distance. 
Moreover, 
$D_{(T,\omega)}(a,a')<D_{(T,\omega)}(a,b)=D_{(T,\omega)}(a',b)$
holds if and only if $aa'|b\in \mathcal R(T)$ and we have 
$D_{(T,\omega)}(a,a')=D_{(T,\omega)}(a,b)=D_{(T,\omega)}(a',b)$
if and only if the $Z$-tree $T|_Z$ is the star-tree on $Z$.

The next result is fundamental for the paper.


\begin{lemma}\label{lem:dist-transfer}
Suppose that $T$ and $T'$ are two $X$-trees with 
equidistant, proper
edge-weightings $\omega$ and $\omega'$, respectively, and let 
$a,a',b\in X$ denote three pairwise distinct elements 
such that $D_{(T,\omega)}(a,a')=D_{(T',\omega')}(a,a')$ and
$D_{(T,\omega)}(a,b)=D_{(T',\omega')}(a,b)$. Then the
following hold:
\begin{enumerate}
\item[(i)] 
If 
$D_{(T,\omega)}(a,a')<D_{(T,\omega)}(a,b)=D_{(T,\omega)}(a',b)$
then
\begin{eqnarray}\label{help1}
D_{(T',\omega')}(a,a')<D_{(T',\omega')}(a,b)=D_{(T',\omega')}(a',b).
\end{eqnarray}
In particular, $D_{(T',\omega')}(a',b)=D_{(T,\omega)}(a',b)$.
\item[(ii)] $aa'|b\in \mathcal R(T)$ if and only if 
$aa'|b\in \mathcal R(T')$. 
\item[(iii)] If $D_{(T,\omega)}(a',b)=D_{(T',\omega')}(a',b)$
then $T|_Z$ is the star-tree on $Z:=\{a,a',b\}$
if and only if $T'|_Z$ is the  star-tree on $Z$.
\end{enumerate}
\end{lemma}
\begin{proof}
(i): 
Note that
$D_{(T,\omega)}(a,a')=D_{(T',\omega')}(a,a')$ combined with
$D_{(T,\omega)}(a,b)=D_{(T',\omega')}(a,b)$  implies that
we cannot have  $D_{(T',\omega')}(a,a')=D_{(T',\omega')}(a,b)$
as this would imply that $D_{(T,\omega)}(a,a')=D_{(T,\omega)}(a,b)$
which is impossible. But then our assumptions imply that
$D_{(T',\omega')}(a,a')=D_{(T',\omega')}(a',b)$ cannot hold either.
Inequality (\ref{help1})  now follows from
the observation preceding the statement of the lemma. 
In particular, this implies
$D_{(T',\omega')}(a',b)=D_{(T',\omega')}(a,b)=D_{(T,\omega)}(a,b)
=D_{(T,\omega)}(a',b)$.

(ii) $\&$ (iii): This is an immediate consequence of (i) and 
the observation preceding the statement of the lemma.
\end{proof}

To be able to state the next result we require  a further definition.
Suppose $\cL\subseteq {X\choose 2}$  and
$T$ is an $X$-tree. Let $x,y\in X$ 
be two distinct leaves of $T$ that are contained in 
the same pseudo-cherry of $T$. Then
we put 
$$
\cL_1(x,y):=\{ab\in \cL\, :\, x\not\in \{a,b\}\}
\cup \{ay\, : \, ax\in \cL\}.
$$
Note that $\cL=\emptyset$ if and only if $\cL_1(x,y)=\emptyset$.

\begin{theorem}\label{theo:first-characterization}
Suppose that $\cL\subseteq {X\choose 2}$ is a set of cords and
that $T$ is an $X$-tree. Let $x,y\in X$ denote
two distinct leaves of $T$ that are
contained in the same pseudo-cherry of $T$. 
Then the following hold.
\begin{enumerate}
\item[(i)] $\cL$ is a topological lasso for $T$ if and only if 
$\cL_1(x,y)\cup \{xy\}$ is a topological lasso for $T$ and $xy\in\cL$.
\item[(ii)] $\cL$ is an equidistant lasso for $T$ if and only if 
$\cL_1(x,y)\cup \{xy\}$ is an equidistant lasso for $T$ and $xy\in\cL$.
\item[(ii)] If $\cl\not=\emptyset$ 
then $\cL$ is a weak lasso for $T$ if and only if 
$\cL_1(x,y)\cup \{xy\}$ is a weak lasso for $T$ and $xy\in\cL$.
\end{enumerate}
\end{theorem}
\begin{proof}
(i): Put $\cL_1^+:=\cL_1(x,y)\cup \{xy\}$ and  
assume first that $\cL$ is a topological lasso for $T$. Then 
$\cL\not=\emptyset$.
Let $T'$ denote an $X$-tree and $\omega$ and $\omega'$  equidistant, proper
edge-weightings for $T$ and $T'$, respectively, so that 
$(T,\omega)$ and $(T',\omega')$ are $\cL_1^+$-isometric. To see that
$T$ and $T'$ are equivalent it clearly suffices to show that 
$(T,\omega)$ and $(T',\omega')$ are also $\cL$-isometric,
that is, 
\begin{eqnarray}\label{eqn:topological-if-and-only-if-1}
D_{(T,\omega)}(a,b)=D_{(T',\omega')}(a,b)
\end{eqnarray}
holds for all $ab\in \cL$.

Suppose $ab\in\cL$. If $x\not\in \{a,b\}$ then $ab\in\cL_1^+$
and so Equation (\ref{eqn:topological-if-and-only-if-1}) holds as 
$(T,\omega)$ and $(T',\omega')$ are $\cL_1^+$-isometric. So assume that
$x\in \{a,b\}$, say $x=a$.  If $y=b$ then $ab=xy\in\cL$ and so
Equation (\ref{eqn:topological-if-and-only-if-1}) holds by the 
same argument. 
If $y\not=b$ then $xb=ab\in \cL$ and so $yb\in\cL_1^+$.
Since $xy\in \cL_1^+$ also holds, we have 
$ D_{(T,\omega)}(x,b)=D_{(T',\omega')}(x,b)$ by Lemma~\ref{lem:dist-transfer}
and so Equation (\ref{eqn:topological-if-and-only-if-1}) follows
in this case, too.
 
Conversely, suppose that $\cL_1^+$ is a topological lasso for $T$
and that $xy\in \cL$. Then $\cL_1^+\not=\emptyset$.
Assume that
$T'$ is an $X$-tree and that $\omega$ and $\omega'$  are equidistant, proper
edge-weightings for $T$ and $T'$, respectively, so that 
$(T,\omega)$ and $(T',\omega')$ are $\cL$-isometric. To see that
$T$ and $T'$ are equivalent it clearly suffices to show that 
$(T,\omega)$ and $(T',\omega')$ are also $\cL_1^+$-isometric,
that is, that
\begin{eqnarray}\label{eqn:topological-if-and-only-if-2}
D_{(T,\omega)}(a,b)=D_{(T',\omega')}(a,b)
\end{eqnarray}
holds for all $ab\in \cL_1^+$.

Suppose $ab\in\cL_1^+$. If $x\in \{a,b\}$ then  $ab=xy\in \cL$ and so
Equation (\ref{eqn:topological-if-and-only-if-2}) holds as
$(T,\omega)$ and $(T',\omega')$ are $\cL$-isometric.
So assume that $x\not\in \{a,b\}$. Then $ab\in \cL_1(x,y)$.
If $y\not\in \{a,b\}$ holds too then $ab\in\cL$ and so 
Equation (\ref{eqn:topological-if-and-only-if-2}) holds by
the same argument. So assume that $y\in \{a,b\}$, say, $y=a$. 
Then $yb=ab\in \cL_1^+$ and so one of $yb\in\cL$
and $xb\in \cL$ must hold by the definition of $\cL_1^+$.
If the former holds then $ab=yb\in \cL$ and so 
 Equation (\ref{eqn:topological-if-and-only-if-2}) holds
by assumption on $(T,\omega)$ and $(T',\omega')$. If $xb\in\cL$ then 
$D_{(T,\omega)}(y,b)=D_{(T',\omega')}(y,b)$
follows by Lemma~\ref{lem:dist-transfer} since, by assumption, 
$xy\in \cL$. But then Equation (\ref{eqn:topological-if-and-only-if-2})
holds in this case, too.

(ii) $\&$ (iii): This follows using similar arguments as in 
the proof of (i).
\end{proof}

\section{The child-edge graph 
}\label{sec:properties-of-G(cL,v)}

In this section we first introduce the child-edge 
graph $G_T(\cL,v)$ associated to an interior vertex $v$ of 
an $X$-tree $T$ and a non-empty set $\cL\subseteq {X\choose 2}$ of cords
and then study some of its properties with regards to corraling an
$X$-tree. We start with a definition.

Suppose $T$ is an $X$-tree, $v\in V^o(T)$, and 
$\cL\subseteq {X\choose 2}$ is a non-empty set of cords. Then we call
the graph $G_T(\cL,v)=(V_{T,v},E_{T,v})$ with vertex set $V_{T,v}$
the set of all child edges of $v$ and edge set $E_{T,v}$ the
set of all $\{e,e'\}\in {V_{T,v}\choose 2}$ for which there
exist leaves $a,b\in X$ such that
$e$ and $e'$ are edges on the path from $a$ to $b$ in $T$ 
and $ab\in \cL$ the {\em child-edge graph (of $v$ with respect 
to $T$ and $\cL$)}. Note that in case there is no danger of 
ambiguity with regards to the $X$-tree $T$ we are referring to, we
 will write $G(\cL,v)$ rather than $G_T(\cL,v)$ and $V_v$ and $E_v$
rather than $V_{T,v}$ and $E_{T,v}$.

For $T$ an $X$-tree, the next 
result provides a key insight into the structure of 
$G(\cL,v)$, $v\in V^o(T)$, whenever $\cL$ corrals $T$. To state it,
we denote by 
$V^o(T)^-$ the set of all interior vertices of $T$
that are not a parent of a pseudo-cherry of $T$.

\begin{lemma}\label{lem:corraling-implies}
Suppose that $T$ is a non-degenerate $X$-tree, that 
 $\cL\subseteq {X\choose 2}$ is a non-empty 
 set of cords that corrals $T$, and that $v\in V^o(T)$. Then 
 the following holds.
\begin{enumerate}
\item[(i)] If $v\in V^o(T)^-$ and $e\in V_v$ is a child edge
of $v$ that is not incident with a leaf of $T$ then
$\{e,e'\}\in E_v$, for all $e'\in V_v$ that are 
incident with a leaf of $T$. In particular, $G(\cL,v)$ 
is connected. 
\item[(ii)] If $v\in V^o(T)-V^o(T)^-$ then $G(\cL,v)$
is connected.
\end{enumerate}
\end{lemma}
\begin{proof}
(i): Let  $e\in V_v$ denote a child edge of $v$ that is not
incident with a leaf of $T$ and let $u\in V(T)$ denote the child 
of $v$ that is incident with $e$. Let 
$\omega: E(T)\to \mathbb R_{\geq 0}$ be an equidistant, proper
edge-weighting for $T$.  Assume for contradiction
that there exists a child $u'\in V(T)$ of $v$ that is a leaf such that
with $e_{u'}$ denoting the child edge of $v$ incident
with $u'$, we have that  $\{e,e_{u'}\}\not\in E_v$. 
Note that since $\omega$ is equidistant and
$u'$ is a leaf of $T$ whereas $u$ is not, 
we must have $\omega(e_{u'})>\omega(e)$.

Assume first that $|V_v|\geq 3$.
Let $T'$ denote the $X$-tree obtained from $T$
by deleting the edge $e_{u'}$ and adding the edge $e^*=\{u,u'\}$.
Clearly, $T'$ is not a refinement of $T$. Consider the
edge-weighting 
$$
\omega': E(T')\to \mathbb R_{\geq 0}:\,\,
f\mapsto\begin{cases}
\omega(f) & \mbox{ if } f\not=e^*,\\
\omega(e_{u'})-\omega(e) & \mbox{ else. } 
\end{cases}
$$
Then it is easy to see that $\omega'$ is equidistant and proper.
Since, by construction, $(T,\omega)$ and $(T',\omega')$
must be $\cL$-isometric it follows that $T'$ is a refinement of 
$T$ as $\cL$ corrals $T$; a contradiction.

Now assume that $|V_v|=2$. If $v\not=\rho_T$ then $v$ must have a parent
$w\in V(T)$. Let $T'$ denote the
$X$-tree obtained from $T$ as before except that we now suppress
$v$ as this has rendered it a vertex with a single child.
Let $\omega'$ be the edge-weighting for $T'$ as defined above
except that we put $\omega'(\{u,w\})=\omega(\{u,v\})+\omega(\{v,w\})$.
Then the same arguments as in the previous case yield a contradiction

If $v=\rho_T$ then let $T'$ denote the
$X$-tree obtained from $T$ by collapsing the edge $\{v,u\}$. Clearly,
$T'$ is not a refinement of $T$. Consider the edge-weighting
$\omega':E(T')\to \mathbb R_{\geq 0}$ for $T'$
defined by putting $\omega'=\omega|_{E(T')}$. Then the same arguments as
in the previous two cases yield a contradiction.

(ii): Assume for contradiction that there 
exists some $v\in V^o(T)-V^o(T)^-$ such that $G(\cL,v)$
is not connected. Then every child of $v$ is a leaf of $T$ and
there exist vertices $e_1, e_2\in V_v$ distinct such that $e_1$ 
and $e_2$ are not joined by a path
in $G(\cL,v)$.  Let $G_1$ and $G_2$ denote the 
connected components of $G(\cL,v)$ containing $e_1$ and $e_2$,
respectively. For all children $u\in V(T) $ of $v$, let $e_u$
denote the child edge of $v$ incident with $u$. Note that since $T$ is
non-degenerate there must exist a vertex
$w\in V(T)$ that is the parent of $v$.
Let $T'$ denote
the $X$-tree obtained from $T$ via the following process. 
Let $i=1,2$. Then, for all $u_i\in V(G_i)$, subdivide
the edge $e_{u_i}$ by a new vertex $p_{u_i}$ not already contained 
in $V(T)$. Next, identify all vertices $p_{u_i}$ into the vertex
$p_i$ and then delete all  copies of the edges $\{v,p_i\}$
from $T$. Finally, add the edge $\{w,p_i\}$ and suppress $p_i$
if $|V(G_i)|=1$. If $V_v=V(G_1)\cup V(G_2)$ then also
suppress the vertex $v$. The resulting tree is $T'$ and,
in either case, $T'$  is clearly not a refinement of $T$. 

Let $\omega$ denote an equidistant, proper edge-weighting for $T$.
Note that $\omega(e_1)=\omega(e)$ must hold for all
$e\in V_v$, as $v\in V^o(T)-V^o(T)^-$. For the following, 
assume first that neither
$p_1$ nor $p_2$ have been suppressed in the construction of
$T'$. Consider the edge-weighting
$$
\omega': E(T')\to \mathbb R_{\geq 0}:\,\,
e\mapsto\begin{cases}
\omega(e) & \mbox{ if } p_1,p_2\not\in e,\\
\omega(\{v,w\}) & \mbox{ if } e\in\{\{w,p_1\},\{w,p_2\} \}\\
\omega(e_1) & \mbox{ else.  } 
\end{cases}
$$
Then, by construction,  $\omega'$ is equidistant and proper
and $(T,\omega)$ and $T',\omega')$ are $\cL$-isometric.
Since, by assumption, $\cL$ corrals $T$ it follows that
$T'$ is a refinement of $T$; a contradiction.

In case one of $p_1$ and $p_2$ or both of them have been 
suppressed in the construction of $T'$ the definition of the
edge-weighting $\omega'$ for $T'$ is similar to the one above
 thus leading  to a contradiction in these cases too.
\end{proof}

The next result is a strengthening of 
Lemma~\ref{lem:corraling-implies}(i).
To state it, we require further terminology concerning child-edge 
graphs. Suppose $T$ is an $X$-tree, $v\in V^o(T)^-$, 
and $\cL\subseteq {X\choose 2}$ is a non-empty set of cords.
Then we denote by $E_l(v)\subseteq E(T)$ the set of child edges of $v$ 
that are incident with a leaf of $T$ and by $E_s(v)\subseteq E(T)$ the
set of child
edges of $v$ that are not contained in $E_l(v)$. Note that $E_s(v)$
is empty if and only if $v$ is the parent of a pseudo-cherry of $T$. 
Also note that $E_l(v)=\emptyset$ might hold. 
Clearly, if neither of them is the empty-set then
$\{E_l(v),E_s(v)\}$ is a partition of $V_v$.
For $v\in V^o(T)^-$, we say that $G(\cL,v)$ is {\em rich} if
the subgraph $G_T(\cL,v)_s$ of $G_T(\cL,v)$ 
induced by $E_s(v)$ is a clique,
and, in case
$E_l(v)\not=\emptyset$, we have for all $e\in E_l(v)$ and 
all $e'\in E_s(v)$ that
$\{e,e'\}\in E_v$. As before we will write $G(\cL,v)_s$ rather
than $G_T(\cL,v)_s$ if there is no ambiguity with regards to which
$X$-tee $T$ we are referring to.
Note that with $T$ and $v$ as above, if $E_l(v)=\emptyset$ then
$G(\cL,v)$ is rich if and only if $G(\cL,v)_s$ is a clique.

To illustrate these concepts consider the set of cords $$\cL=\{ac,ae,ag,bd,be,bh,ce,cg,eh,cd,ef,gh,ai\}$$ on $X=\{a,b,c,d,e,f,g,h,i\}$ and the $X$-tree $T$ depicted in Fig. 2$(i)$. For $v$ as indicated in Fig. 2$(i)$ we depict the child edge graph $G(\cL,v)$ in Fig. 2$(ii)$ which is clearly rich.
\begin{figure}[h]
\begin{center}
\includegraphics[scale=0.5]{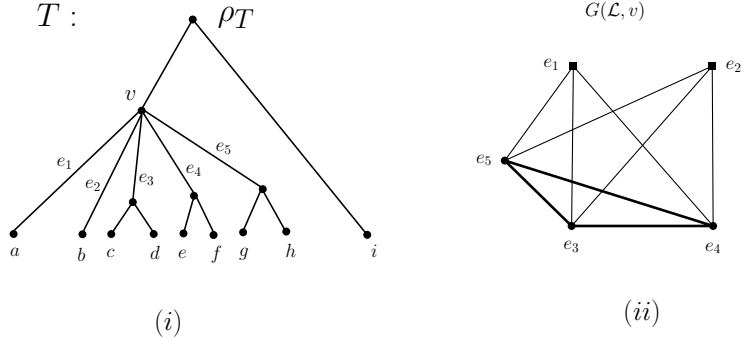}
\end{center}
\label{richExample}
\caption{$(i)$ The $X$-tree $T$ for $X=\{a,b,c,d,e,f,g,h,i\}$. $(ii)$ The child-edge graph $G(\cL,v)$ for $\cL$ as indicated in the text and $T$ and $v$ as in $(i)$. For ease of readability, the vertices in $E_l(v)$ are marked by a square and those in $E_s(v)$ by a dot. The edges of the graph $G(\cL,v)_s$ are represented by thick lines.}
\end{figure}

\begin{proposition}\label{prop:corraling-impliesI}
Suppose that $T$ is a non-degenerate $X$-tree and that
$\cL\subseteq {X\choose 2}$ is a non-empty set of cords. 
If $\cL$ corrals $T$ then $G(\cL,v)$ must be rich,
for all vertices $v\in V^o(T)^-$.
\end{proposition}
\begin{proof}
Let $\omega$ denote an equidistant, proper 
edge-weighting for $T$ and assume for contradiction that 
there exists a vertex $v\in V^o(T)^-$ 
such that $G(\cL,v)$ is not rich. 
 We first show that $G(\cL,v)_s$ 
must be a clique. Suppose $G(\cL,v)_s$ is not a clique, that is,
there exist distinct child edges $e$ and $e'$ of $v$ contained
in $V(G(\cL,v)_s)$ such that with $v_e$ and $v_{e'}$ denoting the
children of $v$ incident with $e$ and $e'$, respectively,
we have, for all $a\in L(v_e)$ and all $b\in L(v_{e'})$ that 
$ab\not\in \cL$. Note that $|V_v|\geq 3$ must hold. Indeed,
if $|V_v|=2$ then $V_v=\{e,e'\}$ and so $\{e,e'\}\in E_v$  since,
by Lemma~\ref{lem:corraling-implies},  
$G(\cL,v)$ is connected. But then $G(\cL,v)$ is a clique and
so $G(\cL,v)_s$ is a clique; a contradiction.
Without loss of generality
assume that $\omega(e)\geq \omega(e')$. 

If $\omega(e)> \omega(e')$
then consider the $X$-tree $T'$ obtained from $T$ by 
deleting the edge $\{v,v_e\}$
and attaching $v_e$ to $v_{e'}$ via the edge $e^*=\{v_e,v_{e'}\}$.
Clearly, $T'$ is not a refinement of $T$. Consider the
edge-weighting $\omega'$ for $T'$ defined by putting
$$
\omega': E(T')\to \mathbb R_{\geq 0}:\,\,
f\mapsto\begin{cases}
\omega(f) & \mbox{ if } f\not=e^*, \\
\omega(e) -\omega(e') &\mbox{ else. }
\end{cases}
$$
Clearly, $\omega'$ is equidistant and proper and, by construction,
$(T,\omega)$ and $(T',\omega')$ are $\cL$-isometric. Since $\cL$
corrals $T$ it follows that $T'$ must be a refinement; a 
contradiction.

If $\omega(e)= \omega(e')$ then consider the $X$-tree 
$T'$ obtained from
$T$ by first identifying the vertices $v_e$ and $v_{e'}$ 
(keeping the label $v_e$) and then deleting one of the edges 
from $v$ to $v_e$. Again, $T'$ is clearly not a refinement of $T$.
Consider the edge-weighting $ \omega': E(T')\to \mathbb R_{\geq 0}$
defined as $\omega'=\omega|_{E(T')}$. Then the same 
arguments as before
imply that $T'$ is a refinement of $T$; a contradiction. Thus,
$G(\cL,v)_s$ is a clique, as required. 

Now if $E_l(v)=\emptyset$ then
$G(\cL,v)$ must be rich. So assume that $E_l(v)\not=\emptyset$.
But then Lemma~\ref{lem:corraling-implies}(i), implies
that $G(\cL,v)$ must be rich.
\end{proof}

We conclude this section with a result that will be useful for establishing
the aforementioned characterization of weak lassos in terms of 
child-edge graphs (Theorem~\ref{theo:characterization-corraling}).
Its proof relies on the well-known fact that an $X$-tee $T'$ is a 
refinement of an $X$-tree $T$ if and only if 
$\mathcal R(T)\subseteq \mathcal R(T')$ (see e.g. 
\cite[Theorem 6.4.1]{SS03}).

\begin{lemma}\label{lem:corraling-unique-cherry}
Suppose that $T$ is an $X$-tree that has a unique cherry $x,y$
and that $\cl\subseteq {X\choose 2}$ is a set of cords that contains the
set
$\{xy\}\cup\{az \, :\, z\in X-\{x,y\} \mbox{ and } a=x \mbox{ or } a=y\}$.
Then $\cL$ corrals $T$.
\end{lemma}
\begin{proof}
By Theorem~\ref{theo:first-characterization}, it suffices to show that
$\cL':=\cL(x,y)\cup\{xy\}$ corrals $T$. Suppose there exists an $X$-tree
$T'$ and equidistant, proper edge-weightings $\omega$ and $\omega'$ of $T$
and $T'$, respectively, such that $(T,\omega)$ and $(T',\omega')$
are $\cL'$-isometric. To see that $T'$ is a refinement of $T$ it suffices
to show that $xy|z\in \mathcal R(T')$ holds for all $z\in X-\{x,y\}$.
Let $z\in X-\{x,y\}$. Then $zy\in \cL'$. Combined with the facts that
$xy\in \cL'$ and that $xy|z\in \mathcal R(T)$ it follows, by 
Lemma~\ref{lem:dist-transfer}, that  $xy|z\in \mathcal R(T')$, as required.
\end{proof}

\section{A characterization of a weak lasso
}\label{sec:corral}

In this section, we characterize sets of cords of $X$ that corral
an $X$-tree $T$ in terms of two properties on the child edge-graphs 
associated to the interior vertex of $T$. In addition, we present
a simple example that illustrates that this characterization does not
hold if the equidistant requirement for the two edge-weightings
mentioned in the definition of such a lasso is dropped.

\begin{theorem}\label{theo:characterization-corraling}
Suppose that $T$ is a non-degenerate $X$-tree and that 
$\cL\subseteq {X\choose 2}$ is a non-empty set of cords. 
Then  $\cL$ is a weak lasso for $T$ if and only if 
the following two properties hold
\begin{enumerate}
\item[(C1)] $G(\cL,v)$ is rich, for all $v\in V^o(T)^-$
\item[(C2)] $G(\cL,v)$ is connected,
for all $v\in V^o(T)-V^o(T)^-$.
\end{enumerate}
\end{theorem}
\begin{proof}
Assume first that $\cL$ corrals $T$. Then Properties~(C1) and (C2)
must hold by Proposition~\ref{prop:corraling-impliesI}
and Lemma~\ref{lem:corraling-implies}(ii).

To see the converse, assume that Properties~(C1) and (C2) hold.
We prove that $\cL$ must be a weak lasso  for $T$
by induction on the size $n$ of $X$. Note that the statement
clearly holds in case  $n=3$ as then $T$ is binary and a refinement
of a binary $X$-tree is the tree itself.

Assume that the statement holds for all finite sets of size $n\geq 3$ and let
$X$ denote a set of size $n+1$. Let $T$ be a non-degenerate $X$-tree and let 
$\cL\subseteq {X\choose 2}$ denote a non-empty set of cords such that 
Properties~(C1) and (C2) are satisfied for $T$ and $\cL$. Note that
$T$ must contain at least one pseudo-cherry. To see
that $\cL$ corrals $T$, let $T'$ denote
an $X$-tree and $\omega$ and $\omega'$ equidistant, proper edge-weightings
for $T$ and $T'$, respectively, such that
$(T,\omega)$ and $(T',\omega')$ are $\cL$-isometric. We 
distinguish between the cases that (i) every
pseudo-cherry of $T$ is in fact a cherry of $T$ and (ii)  that $T$
contains a pseudo-cherry that has at least three leaves.

Case (i): Assume that every pseudo-cherry of $T$ is a cherry and let
$x,y\in X$ such that $x,y$ is a cherry of $T$. 
Note that since $n\geq 4$ and $T$
is non-degenerate, there must exist a vertex $w\in V(T)$
that is the parent of $v_x$ (which is itself the parent of
the cherry $x,y$). Put $X_1=X-\{x\}$ and
$\cL_1=\cL_1(x,y)$. Note that since $G(\cL,v_x)$
is connected, it immediately follows that $xy\in \cL$ and that
since $G(\cL,w)$ is rich we must have $\cL_1\not=\emptyset$. 
Let $T_1$ denote the $X_1$-tree obtained from $T$ by deleting $x$ and its
incident edge and suppressing $v_x$ as this has rendered it a vertex with a
single child. Let 
$$
\omega_1:E(T_1)\to \mathbb R_{\geq 0} : 
e\mapsto\begin{cases}
\omega(e) & \mbox{ if } e\not=\{w,y\} \\
\omega(\{w,v_x\}) +\omega(\{v_x,y\}) &\mbox{ else. }
\end{cases}
$$
Clearly, $\omega_1$ is equidistant and proper and $T_1$ is either the
star-tree on $X_1$ or not. Assume first
that $T_1$ is the star-tree 
on $X_1$. Then $x,y$ is the unique cherry
of $T$ and all elements $z\in X_1-\{y\}$ are adjacent with the root 
$\rho_T$ of $T$ which is $w$. But then
Properties~(C1) and (C2) combined with
 Lemma~\ref{lem:corraling-unique-cherry} imply that $\cL$ corrals $T$.

So assume that $T_1$ is non-degenerate. We claim that
Properties (C1) and (C2) hold for $T_1$ and $\cL_1$. We start with 
establishing Property (C1). 
Assume for contradiction that there exists some
$u\in V^o(T_1)^-$ such that
$G_{T_1}(\cL_1,u)$ is not rich. We first show that 
$G_{T_1}(\cL_1,u)_s$ must be a clique. Assume for contradiction that 
$G_{T_1}(\cL_1,u)_s$ is not a clique, then $|E_s(u)|\geq 2$ and so for all 
$a\in L_{T_1}(u_e)$ and all $b\in L_{T_1}(u_{e'})$
we have that $ab\not\in\cL_1$ where $u_e$ and $u_{e'}$ denote the children
of $u$ in $T_1$ incident with $e$ and $e'$, respectively. Note that
$V^o(T_1)\cup \{v_x\}=V^o(T)$ and that $u=w$ must hold. Indeed 
assume for contradiction that $u\not=w$. 
Then $v_x$ is not a child of $u$ in $T$ and so $u_e$ and $u_{e'}$
are also children of $u$ in $T$. 
Since $G_T(\cL,u)_s$ is a clique by Property~(C1), there must
exist $a\in L_T(u_e)$ and $b\in L_T(u_{e'})$ such that $ab\in \cL$.
If $x\not\in\{a,b\}$ then,
by the definition of $\cL_1$, we have $ab\in\cL_1$; a contradiction.
Thus, $x\in\{a,b\}$. Without
loss of generality assume that $x=a$. Then $y\in L_T(u_e)$ and,
again by the definition of $\cL_1$, we obtain 
$yb\in \cL_1$; a contradiction.
Thus $u=w$, as required. 
But then $y\in \{u_e,u_{e'}\}$ and so one of
$e$ and $e'$ is not a vertex in $G_{T_1}(\cL_1,u)_s$; a
contradiction. Thus, $G_{T_1}(\cL_1,u)_s$ must be a clique, as required. 

Since, by assumption, $G_{T_1}(\cL_1,u)$ is not rich, there must therefore
exist a leaf $z$ of $T_1$ with $e'=\{u,z\}\in E(T_1)$ holding
and some vertex $e$ in $G_{T_1}(\cL_1,u)_s$ such that $\{e,e'\}$ is not
an edge in $G_{T_1}(\cL_1,u)$. Let $u_e$ denote the child of $u$
in $T_1$ incident with $e$. If $u\not=w$ then since the children
of $u$ in $T$ are precisely the children of $u$ in $T_1$
and, by Property~(C1), $G_T(\cL,u)$ is rich we obtain
a contradiction. Thus, $u=w$. But then $y=z$ must hold. 
Since $G_T(\cL,w)$ is rich there must exist some $a\in\{x,y\}$
and some $b\in L(u_e)$ such that $ab\in \cL$.
But then $yb\in \cL_1$ and so $\{e,e'\}$ is an edge in $G_{T_1}(\cL_1,u)$,
a contradiction.

We next establish that Property~(C2) is satisfied by $T_1$ and $\cL_1$
which will conclude the proof of the claim. Let 
$u\in V^o(T_1)-V^o(T_1)^-$. Then $u$ must be the parent of a pseudo-cherry of 
$T_1$.  If $u\not=w$ then since, by assumption, 
every pseudo-cherry of $T$ is a
cherry of $T$ it follows that $u$ is the parent of a cherry of $T_1$.
But then $G_{T_1}(\cL_1,u)$ is connected as $G_{T_1}(\cL_1,u)=G_T(\cL,u)$
and Property~(C2) is satisfied by $T$ and $\cL$. So assume that $u=w$. Then $u$
is the parent of $v_x$ in $T$ and all children of $u$ in $T$ but $v_x$ are
leaves of $T$. Since, by Property~(C1), $G_T(\cL,u)$
is rich, there exists for all children $z$ of $u$ that are leaves of $T$ some
$b_z\in \{x,y\}$ such that $b_zz\in \cL$. But then $yz\in \cL_1$ for all
such children $z$ of $u$ and thus $G_{T_1}(\cL_1,u)$ is connected, 
as required. 
Thus Property (C2) is also satisfied by $T_1$ and $\cL_1$ which completes the
proof of the claim. By induction, it follows that $\cL_1$ is a weak 
lasso for $T_1$.

Let $T'_1$ denote the $X_1$-tree obtained from $T'$ by deleting $x$
and its incident edge and suppressing the parent vertex of $x$ in $T_1'$
if this has rendered it a vertex with a single child. Let $\omega_1'$ 
denote the edge-weighting of $T_1'$ that is canonically induced by $\omega_1$
on the edges of $T_1'$. Then, by Lemma~\ref{lem:dist-transfer} 
combined with the assumption that  $(T,\omega)$ and $(T',\omega')$ are 
$\cL$-isometric, it follows that $(T_1,\omega_1)$ and $(T_1',\omega_1')$ are 
$\cL_1$-isometric. Since $\cL_1$ is a weak lasso for $T_1$ this implies that
$T_1'$ is a refinement for $T_1$. 

To establish that $\cL$ corrals $T$ it
now suffices to show that $xy|a\in \mathcal R(T')$ holds for
all $a\in L(w)-\{x,y\}$. To this end, note that since 
Property (C1) is satisfied
by $T$ and $\cL$, we have for all children $a\in L(w)$ that are
leaves of $T$ that there exists some $b\in L_T(v_x)=\{x,y\}$ such that
$ab\in \cL$. Combined with Lemma~\ref{lem:dist-transfer}
and $xy\in \cL$, it follows that 
$D_{(T,\omega)}(x,a)=D_{(T',\omega')}(x,a)=
D_{(T,\omega)}(y,a)=D_{(T',\omega')}(y,a)$. Since $xy|a\in\mathcal R(T)$
we obtain $xy|a\in\mathcal R(T')$, as required. This completes
the proof of the induction step in this case.

Case (ii): Assume that $T$ contains a pseudo-cherry $\frak c$ 
of size three or more. Let $v\in V^o(T)$ denote the parent of 
$\frak c$ in $T$. Then, by Property~(C2), $G(\cL,v)$ is connected. 
Since in any connected graph there exists a vertex whose removal (plus incident edges) leaves the graph connected (see \cite[Proposition 1.4.1]{Diestel05} where a more general result is established), it follows that we may choose some $x\in X$
such that the graph $G^-(\cL,v)$ obtained from $G(\cL,v)$ by deleting
$e_x$ and its incident edges is connected. Note that since $x$ is
a leaf in $\frak c$, we have $v=v_x$.

Put $X_1:=X-\{x\}$, choose some $y\in L(v_x)$ such that $xy\in \cL$, and put
$\cL_1:=\cL(x,y)$. Consider the
$X$-tree $T_1$ obtained from $T$ by deleting $x$ and its incident
edge. 
We claim again that $G_{T_1}(\cL_1,u)$ satisfies Properties~(C1) and (C2)
for all $u\in V^o(T_1)$ as specified in those conditions. To see this, 
note first that since $v_x$ has at least two children in $T_1$, we have
 $V^o(T_1)= V^o(T)$ and so also $V^o(T_1)^-= V^o(T)^-$.

We start with establishing Property~(C1).
Assume for contradiction that there exists
some vertex $u\in  V^o(T_1)^-$ such that $G_{T_1}(\cL_1,u)$ is not 
rich. We again show first that $G_{T_1}(\cL_1,u)_s$ is a clique. Assume for
contradiction that this is not the case and let $e$, $e'$, $u_e$, and 
$u_{e'}$ be as in the corresponding situation in the previous case. Note that
since $u\in V^o(T_1)^-= V^o(T)^-$ and
$v_x\in V^o(T)-V^o(T)^-$
we have $v_x\not=u$. 
But then similar arguments as the ones used to show 
in Case (i) that $w\not=u$ in the context of establishing
that $G_{T_1}(\cL_1,u)_s$ is a clique yield a contradiction.
Thus, $G_{T_1}(\cL_1,u)_s$ must be a clique, as required.
As in the previous
case, there must therefore 
exist some $z\in X_1$ such that $e=\{z,u\}\in E(T_1)$ and some
vertex $e'$ in $G_{T_1}(\cL_1,u)_s$ such that $\{e,e'\}$ is not an edge in
$G_{T_1}(\cL_1,u)$. Note that the same arguments as above imply that 
$u\not=v_x$.
By the definition of $\cL_1$, it follows that
$x$ must be the unique leaf in $L_T(u_{e'})$ such that
$xz\in \cL$. Since $y$ and $x$ are leaves in the same pseudo-cherry of $T$,
we obtain $yz\in\cL_1$. Consequently, $\{e,e'\}$ is an edge in 
$G_{T_1}(\cL_1,u)$, a contradiction. Thus, Property~(C1) is satisfied by
$T_1$ and $\cL_1$.

To see that $T_1$ and $\cL_1$ satisfy Property~(C2)
assume without loss of generality that there exists some 
$u\in V^o(T)-V^o(T)^-$ such that $G_{T_1}(\cL_1,u)$ is not connected.
Then $u\not=v_x$, by the choice of $x$. Since $u$ is the parent of 
a pseudo-cherry in $T$ it follows that $u$ must be the
parent of the same pseudo-cherry in $T_1$. But then 
$G_{T_1}(\cL_1,u)=G_T(\cL,u)$ and so $G_{T_1}(\cL_1,u)$
must be connected as, by Property (C2), $G_T(\cL,u)$ is connected;
a contradiction.  This concludes the proof of the claim.
By induction, it follows that $\cL_1$ corrals $T_1$.

Let $T'_1$, $w$, and $\omega_1'$ be as in the previous case. Then, as in that 
case, $T_1'$ must be a refinement of $T_1$. 
We claim that
for all children $u$ of $w$ in $T$ distinct from $v_x$ and all $a\in L_T(u)$
we must have 
$$
xy|a\in \mathcal R(T').
$$

To see this, note first that Property (C1) implies 
for all such children $u$ of $w$ 
that there
must exist some $a_u\in L_T(u)$ and some $z_{v_x}\in L_T(v_x)$ such that 
$a_uz_{v_x}\in \cL$. Let $u$ denote a child of $w$ in $T$ distinct 
from $v_x$ and put $a= a_u$ and  $z=z_{v_x}$. We show first that 
\begin{eqnarray}\label{eqn:weak-lasso-equation}
D_{(T,\omega)}(a,y)= D_{(T',\omega')}(a,y) \mbox{ and }
D_{(T,\omega)}(a,x)= D_{(T',\omega')}(a,x).
\end{eqnarray}
Clearly, if $z\in\{x,y\}$  then one of the two equations in
(\ref{eqn:weak-lasso-equation}) must hold. 
Assume without loss of generality that $z\not=x$. 
Since $G_T(\cL,v_x)$ is connected by Property~(C2), it follows that
there exists a path $z=z_1,z_2,\ldots,z_k=x$, $k\geq 2$, from $z$ to $x$
in $G_T(\cL,v_x)$. But then 
Lemma~\ref{lem:dist-transfer}(i) implies  
$D_{(T,\omega)}(a,x)= D_{(T',\omega')}(a,x)$ since, for all $1\leq i\leq k-1$,
we have $z_iz_{i+1}\in \cL$. If $z\not=y$ then similar arguments
imply that $D_{(T,\omega)}(a,y)= D_{(T',\omega')}(a,y)$. 
If $z=y$ then $ya=za\in \cL$
and so $D_{(T,\omega)}(a,y)= D_{(T',\omega')}(a,y)$.
Thus both equations in (\ref{eqn:weak-lasso-equation}) must hold,
as required.
Combined with  $xy\in\cL$ (which holds by the choice of $y$) and
the fact that $xy|a\in \mathcal R(T)$ we obtain 
$$
xy|a\in \mathcal R(T')
$$
in view of Lemma~\ref{lem:dist-transfer}(ii). Thus the claim
follows if $|L_T(u)|=1$. 
So assume that $|L_T(u)|\geq 2$. Suppose $a'\in L_T(u)-\{a\}$.
Then $y|a'a\in \mathcal R(T_1')$ must hold as $T_1'$ is a refinement
of $T_1$ and $y|a'a\in \mathcal R(T_1)$. Since the only 
$\{x,y,a',a\}$-tree that can simultaneously display the triplets 
$y|a'a$ and $xy|a$ is the tree with cherries $x,y$ and $a',a$
and that tree is equivalent with the tree $T|_{\{x,y,a',a\}}$
it follows that $xy|a'\in \mathcal R(T')$, as
claimed.

Combined with the fact that 
$T'_1$ is a refinement of $T_1$ it follows that $T'$ is a refinement of $T$.
Hence, $\cL$ corrals $T$ which concludes the proof of the induction step
in this case too and, thus, the proof of the theorem.

\end{proof}

Note that Theorem~\ref{theo:characterization-corraling} 
immediately implies that for
a non-empty set $\cL$ of cords of $X$ to be a weak lasso
for a non-degenerate $X$-tree $T$, it must 
 be a {\em covering} of $X$, that is, 
$X=\bigcup_{A\in \mathcal L}A$.
Also note that Theorem~\ref{theo:characterization-corraling} immediately  implies that a minimum size weak lasso for $T$ 
must have
$$\sum_{v\in V^o(T)^-}({|V(G(\cL,v)_s)|\choose 2} +
|V(G(\cL,v)_s)|\times |V(G(\cL,v)_l)|)+\sum_{v\in V^o(T)-V^o(T)^-}|V_v|$$
cords. Thus, such a lasso 
has at most $|V^o(T)^-|{m\choose 2} +(m-1)|V^o(T)-V^o(T)^-|$
cords where $m=\max_{v\in V^o(T)^-}|V_v|$ and at least 
$(l-1)|V^o(T)|$ cords where $l=\min_{v\in V^o(T)^-}|V_v|$. Note that these
bounds are sharp in the case that all interior vertices of $T$ have the
same number $k$ of children. In the former case $T$ is such that
no interior vertex of $T$
that is not a parent of a pseudo-cherry is adjacent with a leaf of $T$.
In the latter case $T$ is the {\em bearded caterpillar tree} on $X$, that is,
$T$ is a (rooted) path and every vertex of that path is adjacent with
$k-1$ leaves except for the end vertex of $T$ that is not $\rho_T$ which has
$k$ children. 

Finally,  note that as the example presented 
in Fig.~\ref{corraling-counterexample} illustrates,
for Theorem~\ref{theo:characterization-topology} to hold 
the requirement that both proper edge-weightings are equidistant
in the definition of a weak lasso 
cannot be dropped. 
More precisely for $X=\{a,b,c,d,e,f\}$ and
$\cL=\{ab,bc,bd,af,ae\}$ the $X$-tree $T$ pictured on the left of that figure
(with the indicated edge-weighting $\omega$ ignored for the moment)
satisfies Properties (C1) and (C2) and the $X$-tree $T'$ depicted on
the right of that figure (again with the indicated edge-weighting $\omega'$
ignored where $0<\epsilon<1$) is clearly not a refinement of $T$. 
Also $\omega$ and $\omega'$
are obviously proper and, in the case of $\omega$,
equidistant and $(T,\omega)$ and $(T',\omega')$ 
are $\cL$-isometric.

\begin{figure}[h]
\begin{center}
\includegraphics[scale=0.25]{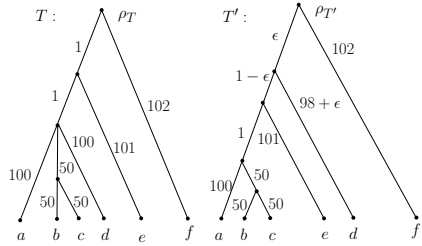}
\caption{\label{corraling-counterexample}
An example illustrating that for 
Theorem~\ref{theo:characterization-corraling} 
to hold the equidistance requirement in the 
definition of a weak lasso cannot be dropped (see text for details).}
\end{center}
\end{figure}

\section{A characterization of an equidistant lasso
}\label{sec:edge-weight}

In this section, we present a characterization of an equidistant
lasso $\cL\subseteq {X\choose 2}$ for an $X$-tree $T$ in terms 
of the child-edge graphs associated to the interior vertices of $T$. 
To establish it, we require a further
notation. Suppose $T$ is an $X$-tree and $v$ and $w$
are two vertices of $T$.
Then we denote by $E_T(v,w)$ the set of all edges of $T$
on the path from $v$ to $w$.

\begin{theorem}\label{theo:edge-weight}
Suppose $T$ is an $X$-tree and 
$\cL\subseteq {X\choose 2}$ is a non-empty set of cords.
Then the following are equivalent:
\begin{enumerate}
\item[(i)] $\cL$ is an equidistant lasso for $T$.
\item[(ii)] for every vertex $v\in V^o(T)$, the graph
$G(\cL,v)$ contains at least one edge.
\end{enumerate}
\end{theorem}
\begin{proof}
(i) $\Rightarrow$ (ii): Suppose $\cL$ is an equidistant 
lasso for $T$ and assume for contradiction that there
exists an interior vertex $v$ of $T$ for which $G(\cL,v)$ 
does not contain an edge. 
Note first that $v=\rho_T$ cannot hold. Indeed, suppose
 $\omega_1$ is a equidistant, proper edge-weighting for $T$ 
and let $\epsilon>0$
be a sufficiently small real number. Consider the edge-weighting
 $\omega_2:E(T)\to \mathbb R_{\geq 0}$ defined by putting, 
for all $e\in E(T)$, $\omega_2(e)=\omega_1(e)-\epsilon$
if $\rho_T\in e$ and $\omega_2(e)=\omega_1(e)$ 
else. Clearly, $\omega_2$ is an equidistant, proper edge-weighting
for $T$ distinct from $\omega_1$ and  $(T,\omega_1)$ and $(T,\omega_2)$
are $\cL$-isometric. Since, by assumption,
$\cL$ is an equidistant lasso for $T$ it follows that
$\omega_1=\omega_2$; a contradiction.
Thus, $v\not=\rho_T$  and so there must exist a parent $w\in V(T)$ of 
$v$ in $T$. Let $\omega_1$
denote again an equidistant, proper edge-weighting for $T$. Consider 
the map $\omega_2:E(T)\to \mathbb R_{\geq 0}$ defined
by putting 
\begin{equation*}
\omega_2(e) =
\begin{cases}
\omega_1(e) &\mbox{ if } v\not\in e,\\
\omega_1(e)-\epsilon &\mbox{ if } e=\{v,w\},\\
\omega_1(e)+\epsilon & \mbox{ else, }
\end{cases}
\end{equation*}
where  $\epsilon>0$ is small enough.  Clearly, $\omega_2$ is an equidistant, 
proper edge-weighting for $T$ which is distinct from $\omega_1$. By 
construction, $(T,\omega_1)$ and $(T,\omega_2)$ are 
$\cL$-isometric and so $\omega_1=\omega_2$ must hold as
$\cL$ is an equidistant lasso for $T$; a contradiction. 

(ii) $\Rightarrow$ (i): Suppose that, for all $v\in V^o(T)$,
the graph $G(\cL,v)$ has at least one edge and assume for
contradiction that $\cL$ is not an equidistant lasso for
$T$. Then there exist distinct equidistant, proper 
edge-weightings $\omega_1$ and $\omega_2$
for $T$ such $(T,\omega_1)$ and $(T,\omega_2)$ are $\cL$-isometric.
 Thus, there must exist some edge $e\in E(T)$ such
that $\omega_1(e)\not=\omega_2(e)$. 


Assume without loss of generality that
$e=\{v,v_0\}$ is such that with $v$ being the parent of $v_0$
we have that $\omega_1(e')=\omega_2(e')$
holds for all edges $e'$ of $T$ that lie on a path from 
$v_0$ to a leaf of $T$ contained in $L(v_0)$. Note that $v_0$ could be a leaf
of $T$ in which case such a path has length zero.
Let $v_1,v_2,\ldots, v_l$, $l\geq 1$ denote
the other children of $v$. Then, by assumption, $G(\cL,v)$ contains
at least one edge and so there exist
$i,j\in\{0,l\}$ distinct and leaves $y_i,y_j\in L(v)$ 
that are descendants of $v_i$ and $v_j$ (or coincide with them), 
respectively, such that $y_iy_j\in \cL$. Let $z\in L(v_0)$ denote
a leaf of $T$. By Property~(E1), we obtain
\begin{eqnarray*}
\omega_1(e)+\sum_{e'\in E_T(v_0,z)}\omega_1(e')&=&D_{(T,\omega_1)}(v,z)
= \frac{1}{2}D_{(T,\omega_1)}(y_i,y_j)=\frac{1}{2}D_{(T,\omega_2)}(y_i,y_j)\\
&=&D_{(T,\omega_2)}(v,z)=\sum_{e'\in E_T(v_0,z)}\omega_2(e')+\omega_2(e).
\end{eqnarray*}
and so $\omega_1(e)=\omega_2(e)$ follows by the choice of $e$; a contradiction.
Thus, $\cL $ must be an equidistant lasso for $T$.
\end{proof}

Theorem~\ref{theo:edge-weight} immediately implies that
an equidistant lasso on $X$ 
need not be a covering of $X$.
Also, it implies that the size of a minimum 
equidistant lasso for an $X$-tree $T$ is $|V^{\circ}(T)|$. For the extreme 
cases that $T$ is the star tree on $X$ such a lasso has 
precisely one element and
if $T$ is binary such a lasso has $|X|-1$ elements as any 
such tree is known to have $|X|-1$ interior vertices 
(see e.\,g.\,\cite{SS03}).

Theorem~\ref{theo:characterization-corraling} combined with
Theorem~\ref{theo:edge-weight} immediately implies the
following link between equidistant and weak lassos.

\begin{corollary}\label{cor:weak-implies-equidistant}
Suppose $T$ is an $X$-tree 
and  $\cL\subseteq {X\choose 2}$ is a non-empty set of cords. 
Then  $\cL$ is an equidistant lasso for $T$ whenever it is a weak
lasso for $T$. 
\end{corollary}

We remark in passing 
that dropping the requirement 
that the two edge-weightings have to be 
equidistant in the definition of
an equidistant lasso  
gives rise to the definition of
an {\em edge-weight lasso} for an $X$-tree.
However it is easy to see that Theorem~\ref{theo:edge-weight}
does not hold with equidistant lasso replaced by edge-weight lasso.


\section{A characterization of a topological lasso
} \label{sec:topological}

In this section, we prove the companion result for when a set 
of cords is a topological lasso for an $X$-tree $T$ in terms 
of the child-edge graphs associated to the interior vertices of $T$. 
We start again with some more notation. 

Suppose that $T$ is an $X$-tree and that $v\in V^o(T)$ but not
the root of $T$. Then we
denote the $L(v)$-tree obtained from $T$ by deleting the parent
edge of $v$ by $T_{L(v)}$. Now suppose that $Y\subsetneq X$ is such
that there exists some $v\in V^o(T)$ such that $Y=L(v)$. Then 
we denote the root of $T_Y$ by $\rho(T_Y)$.

\begin{theorem}\label{theo:characterization-topology}
Suppose $T$ is an $X$-tree and 
$\cL\subseteq {X\choose 2}$ is a non-empty set of cords. 
Then the following are equivalent:
\begin{enumerate}
\item[(i)] $\cL$ is a topological lasso for $T$.
\item[(ii)] for every vertex $v\in V^o(T)$, the graph
$G(\cL,v)$ is a clique.
\end{enumerate}
\end{theorem}

\begin{proof}
(i) $\Rightarrow$ (ii): Suppose that $\cL$ is a topological lasso 
for $T$ and assume for contradiction that there exists some 
vertex $v\in V^o(T)$ such that $G(\cL,v)$ is not a clique. Then there
exist child edges $e,e'\in E(T)$ of $v$ such that $\{e,e'\}\not\in E_v$.
Let $v_e$ and $v_{e'}$ denote the children of $v$ incident with $e$ and
$e'$, respectively, and let $\omega$ denote an equidistant, proper 
edge-weighting of $T$. We distinguish the cases that (i) $v$ is the
parent of a pseudo-cherry of $T$ and that (ii) $v\in V^o(T)^-$.

Case (i): Assume that $v$ is the parent of a pseudo-cherry of $T$.  Then since 
every topological lasso for $T$ is in particular a weak lasso for $T$,
Theorem~\ref{theo:characterization-corraling} implies that $G(\cL,v)$
is connected and that, in addition to $x:=v_e$ and $y:=v_{e'}$,  
there must exists a further child of $v$ (that is a leaf of $T$). 
Note that since $\omega$ is equidistant, we must have 
$\omega(e)=\omega(e')$. 

Let $T'$ denote the $X$-tree obtained from $T$ by 
subdividing the edge $e'$ by a vertex $w$ that is not already contained
in $V(T)$, adding the edge $\{x,w\}$ and deleting the edge $e$. 
Clearly, $T$ and  $T'$ are not equivalent.
Let $0<\epsilon<\omega(e)$ and consider edge-weighting 
$$
\omega': E(T')\to \mathbb R_{\geq 0}:\,\,
f\mapsto\begin{cases}
\omega(f) & \mbox{ if } w\not\in f,\\
\omega(f)-\epsilon &\mbox{ if } f\in\{\{w,x\},\{w,y\}\},\\
\epsilon &\mbox{ else. }
\end{cases}
$$
Clearly, $\omega'$ is equidistant
and proper and it is straight forward to see that 
$(T',\omega')$ and $(T,\omega)$ are $\cL$-isometric. Since $\cL$
is a topological lasso for $T$ it follows that 
$T$ and $T'$ must be equivalent, a contradiction.

Case (ii): Assume that $v\in V^o(T)^-$.
Then since 
every topological lasso for $T$ is in particular a weak lasso for $T$,
Theorem~\ref{theo:characterization-corraling} implies that $G(\cL,v)$
is rich. But then $v_e$ and $v_{e'}$ must be leaves of $T$. With
$x=v_e$ and $y=v_{e'}$ we obtain a contradiction using the same 
arguments as in Case (i).

(ii) $\Rightarrow$ (i): Suppose that, for every vertex $v\in V^o(T)$, 
the graph
$G(\cL,v)$ is a clique and assume that $T'$ is an $X$-tree 
and $\omega$ and $\omega'$ are equidistant, proper edge-weightings
for $T$ and $T'$, respectively, such that $(T,\omega)$ and
$(T,\omega')$ are  $\cL$-isometric. If $T$ is the star-tree on $X$ 
then  $\cL$ must necessarily be a topological lasso for $T$. So assume that
$T$ is non-degenerate. Then, by 
Theorem~\ref{theo:characterization-corraling}, $\cL$ is a 
weak lasso for $T$ and so $T'$ must be a refinement of $T$. 

We next
show that $T$ and $T'$ are in fact equivalent. Assume for contradiction
that $T$ and $T'$ are not equivalent. Then there must exist a non-empty subset
$Y\subsetneq X$ of leaves of $T$ such that the subtree $T_Y$ of $T$ with
leaf set $Y$ is equivalent with the subtree $T'_Y$ of $T'$ with leaf set
$Y$ but the subtrees of $T$ and $T'$ with root the parents of $\rho(T_Y)$ and 
$\rho(T'_Y)$, respectively, are not. Let $w$ denote the parent of 
$\rho(T_Y)$ in $T$ and $w'$ the parent of $\rho(T'_Y)$ in $T'$. Then
there must exist some $z\in L(w')-L(\rho(T'_Y))$ and  distinct
$x,y\in L(\rho(T'_Y))=Y$ such that $xy|z\in \mathcal R(T')$ and
$xy|z\not\in \mathcal R(T)$. 
Hence, $\rho(T_Y)$ must
lie on the path from $x$ to $y$ in $T$. Combined with the fact that
$T'$ is a refinement of $T$,
it follows that $T|_{\{x,y,z\}}$ is the star-tree on $Z:=\{x,y,z\}$ whose
unique interior vertex is $\rho(T_Y)$. Since $\omega$ is equidistant,
we obtain
\begin{eqnarray}\label{eqn:equality}
D_{(T,\omega)}(x,y)=D_{(T,\omega)}(x,z)=D_{(T,\omega)}(z,y).
\end{eqnarray}
Let $e_x,e_y,e_z\in E(T)$ denote the child edges of $\rho(T_Y)$ that are
crossed by a path from $\rho(T_Y)$ to $x$, $y$, and $z$, respectively,
and let  $v_{e_s}$ denote the child of $\rho(T_Y)$
 incident with $e_s$, for all $s\in Z$.
Since, by assumption, $G(\cL,\rho(T_Y))$ is a clique there must 
exist leaves $a\in L(v_{e_x})$, $b\in L(v_{e_y})$, and $c\in L(v_{e_z})$
such that $ab,bc,ca\in \cL$.
By the same reason, it follows in view of 
Lemma~\ref{lem:dist-transfer} that 
$ D_{(T,\omega)}(p,q)=D_{(T',\omega')}(p,q)$ must hold for any two elements
$p,q\in Z$ distinct. Combined with Equality (\ref{eqn:equality}),
we obtain $ D_{(T',\omega')}(p,q)=D_{(T',\omega')}(p,q)$,
for any two such elements $p$ and $q$. But then 
$T'|_Z$ must be the star-tree on $Z$ and so $xy|z\not\in \mathcal R(T')$,
a contradiction.
\end{proof}

Note that 
Theorem~\ref{theo:characterization-topology} immediately
implies that a topological lasso $\cL\subseteq {X\choose 2}$  
must be a covering of $X$. Also note that Theorem~\ref{theo:characterization-topology} implies that if $\cL$ is a
topological lasso for an $X$-tree $T$ then $\cL$ must
contain at least $\sum_{v\in V^o(T)}{|V_v|\choose 2}$ cords
and that $\cL= {X\choose 2}$ must hold in case $T$ is the star-tree on $X$.
Finally, note that as the example presented in 
Fig.~\ref{fig:top-counterexample} shows, 
the requirement that the two proper edge-weightings
in the definition of a topological lasso must be equidistant
cannot be dropped. More precisely for $X=\{a,b,c,d\}$ and
$\cL=\{ab,cd, ad\}$ the $X$-tree $T$ pictured on the left of that figure
(with the indicated edge-weighting $\omega$ ignored for the moment)
satisfies Properties (C1) and (C2) and is not equivalent
with the $X$-tree $T'$ depicted on
the right of that figure (again with the indicated edge-weighting $\omega'$
ignored where $0<\epsilon<1$). Also $\omega$ and $\omega'$
are clearly proper and, in the case of $\omega$, also 
equidistant and $(T,\omega)$ and $(T',\omega')$ are  $\cL$-isometric.

\begin{figure}[h]
\begin{center}
\includegraphics[scale=0.3]{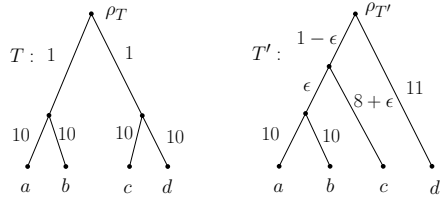}
\caption{\label{fig:top-counterexample}
An example showing that for 
Theorem~\ref{theo:characterization-topology} 
to hold the ``equidistant'' requirement in the 
definition of a topological lasso cannot be dropped (see text for details).}
\end{center}
\end{figure}

Combining Proposition~\ref{theo:edge-weight} and 
Theorem~\ref{theo:characterization-topology},
we obtain the following corollary.

\begin{corollary}\label{cor:coincide}
Suppose that $T$ is an $X$-tree.
Then every topological lasso 
for $T$ must also be an edge-weight lasso for $T$. 
Moreover, if $T$ is binary and $\cL\subseteq {X\choose 2}$ 
then $\cL$ is a topological lasso for $T$ if and only if
$\cL$ is an equidistant lasso for $T$. In particular, every
topological/equidistant lasso for $T$ is also a strong lasso for $T$
in this case.
\end{corollary}

Note however that for $\cL\subseteq {X\choose 2}$ a set
of cords and $T$ an $X$-tree it is possible that $\cL$ is
simultaneously an edge-weight and a topological lasso for 
$T$ but $T$ is not binary. We present such an example in Section \ref{sec:examples}.

\section{Examples of lassos} \label{sec:examples}

In this section we apply our findings to two types of constructions
of sets of cords of $X$. Both of them were originally introduced 
in \cite{DHS12} 
for the case of edge-weighted, unrooted, phylogenetic trees with
leaf set $X$ where an edge-weighting  for such a tree is 
defined as in the case of an $X$-tree. 

Assume for the remainder of this section that $T$ is an $X$-tree.
Then the first example relies on the notion of a circular
ordering (of the leaf set) of $T$ (see e.g. \cite{SS03}
for further details on such orderings). Following \cite{DHS12} a {\em
circular ordering (of $X$) of $T$} is a cyclic 
permutation $\sigma$ of $X$ such that the following holds.
There exists a planar embedding of $T$ such that, for every $x\in X$,
the leaf that is encountered after $x$ when traversing $T$ in
that embedding in say, a counter-clockwise fashion, is the leaf $\sigma(x)$.
For example, $(a,b,c,d)$ is a circular ordering of $T$ the
$\{a,b,c,d\}$-tree depicted in Figure~\ref{fig:top-counterexample}(a).

Let $(x_1,x_2,\ldots,x_n)$
denote a circular ordering of the leaves of $T$
where $n:=|X|$ and put $x_{n+1}=x_1$. Then since for every interior 
vertex $v$ of $T$ there exists some $i\in \{1,\ldots,n\}$ 
such that $v$ lies on the path from $x_i$ to $x_{i+1}$, 
Theorem~\ref{theo:edge-weight} implies that
the set 
$$
\cL_c=\{x_ix_{i+1}\,:\, 1\leq i\leq n\}
$$ is an
equidistant lasso for $T$. However, $\cL_c$ is clearly not an
equidistant lasso of minimal size.
In view of Theorem~\ref{theo:characterization-topology}, 
$\cL_c$ is a topological lasso
for $T$ if and only if every interior vertex $v\in V^o(T)$ has degree
three (except possibly the root $\rho_T$ which might also have
degree two) as in that case $G(\cL_c,w)$ is a complete graph
for all $w\in V^o(T)$. Thus, $\cL_c$ is also a strong lasso
for such $X$-trees. In view
of Theorem~\ref{theo:characterization-corraling},
$\cL_c$ is a weak lasso for $T$ if and only if  
every vertex $v\in V^o(T)^-$ has degree
three (except possibly the root $\rho_T$ which might also have
degree two) as in that case $G(\cL_c,w)$ is rich
for all $w\in V^o(T)^-$ and connected for all 
$w\in V^o(T)- V^o(T)^-$.

Our final construction relies on the notion of a bipartition $\{A,B\}$ of $X$ 
and
was introduced in \cite{DHS12} where it was shown that
the set 
$$
A\vee B:=\{ab\in{X\choose 2}\,:\,a\in A\mbox{ and } b\in B\}
$$
is a topological lasso\footnote{The definition of a topological lasso
for an unrooted phylogenetic tree on $X$ is the same as that of a
topological lasso for an $X$-tree but with the requirement dropped 
that the two proper edge-weightings mentioned in that definition 
are equidistant.} 
for an unrooted phylogenetic tree $T'$ with leaf set
$X$ of size 4 or more if and only if, for every 2-subset $\frak c$ of $X$ 
whose elements have that same parent in $T'$,
we have that $A\cap\frak c\not=\emptyset\not=B\cap \frak c$.
Note that this implies in particular
that for an unrooted phylogenetic tree on $X$ to be topologically
lassoed by $A\vee B$, every interior vertex of $T$ can be adjacent 
with at most two leaves. Thus every
pseudo-cherry of $T'$ (defined as in
the case of an $X$-tree) must be a cherry of $T$. 

Defining for an unrooted phylogenetic tree $T'$
on $X$ a set $\cL\subseteq {X\choose 2}$ of cords to be an edge-weight lasso
as in the case of an $X$-tree but again with the requirement
``equidistant'' on the proper edge-weightings removed, it
is not difficult to see that $A\vee B$ is not an edge-weight lasso
for $T'$. Also it is straight forward to see that 
any two proper edge-weightings $\omega$ and $\omega'$ for
$T'$ such that 
$D_{(T',\omega)}(a,b)=D_{(T',\omega')}(a,b)$ holds for all $ab\in \cL$
must coincide on the interior edges of $T'$ where for a
proper edge-weighting $\alpha$ of $T'$ we denote the induced
distance on $V(T')$ also by $D_{(T',\alpha)}$.

In the case of $X$-trees the situation changes in so far that if
$T$ is non-degenerate and 
$\{A,B\}$ is such that every pseudo-cherry of $T$ contains elements
from both $A$ and $B$ then, in view
of Theorem~\ref{theo:characterization-corraling}, $A\vee B$
must be a weak lasso for $T$ and thus, by 
Corollary~\ref{cor:weak-implies-equidistant}, also an
equidistant lasso for $T$.  
In view of Theorem~\ref{theo:characterization-topology}, $A\vee B$ is not a 
topological lasso
for $T$ unless every interior vertex $v\in V^o(T)$ of $T$
is incident with at most two leaves of $T$ and, if $v$ is incident with
two leaves, then one is contained in $A$ and the other in $B$
as otherwise $G(A\vee B, v)$ would not be a clique. Since for such $X$-trees
$T$ we have, for all $v\in V^o(T)$, that $G(\cL,v)$ contains at least one
edge it follows that $A\vee B$ is a strong lasso for $T$.

If $T$ is the star-tree on $X$ then $A\vee B$ is a weak lasso for $T$. 
Also, $A\vee B$ is an equidistant lasso for $T$ as $G(\cL,\rho_T)$ 
contains at least one edge but it is not a topological lasso for $T$
as 
$A\vee B\not={X\choose 2}$.

\section{Conclusion} \label{sec:conclusion}
In the form of investigating when a set of cords of a finite set
$X$ of size at least three is an 
equidistant/topological/weak/strong lasso for an $X$-tree, we have 
addressed the topical problem of when a set of partial 
distances for a set of individuals within a variety uniquely 
determines a dendogram on those individuals. Such structures 
are commonly constructed as part of a phenetic clustering 
step within a genomewide association study to better 
understand the link between phenotypic and a genotypic variation
within a variety. 
For $T$ an $X$-tree and $\cL\subseteq {X\choose 2}$ a 
set of cords, we have presented characterizations for when $\cL$
is an equidistant/weak/tropological
lasso for $T$ in terms of the structure  of the
child-edge graphs associated to the 
interior vertices of $T$. 
As immediate consequences, our characterizations allow us
to not only shed light into the problem of when two of 
the above types of lassos coincide but also into the size of minimum size
equidistant/topological/weak lassos. 

Despite these encouraging results a number of open 
questions remain. For example, the characterizations
above require knowledge of the structure of $T$ in the
form of the child-edge graphs associated to the interior
vertices of $T$. Thus, is it possible to characterize
or at least understand lassos without this structural insight
into $T$. A potential
candidate for this might be the graph $\Gamma(\cL)$
associated to $\cL$ whose vertex set is $X$ and whose
edge set is $\cL$. The underlying rational
for this is that for $|X|\geq 4$,  it was shown in
\cite[Theorem 1]{DHS12} that for $\cL$ to be a topological lasso for an
unrooted phylogenetic tree $T$ 
on $X$ the graph $\Gamma(\cL)$ has to be connected
and for $\cL$ to be an edge-weight lasso for $T$ it
has to be strongly non-bipartite 
(where a graph $G$  is said to be {\em strongly non-bipartite}
if every connected component of $G$ is not bipartite).
Also for $\cL$ as constructed in the first example
in Section~\ref{sec:examples} the graph $\Gamma(\cL)$ is connected.

To overcome the potential loss of information 
in distance based phylogenetic tree reconstruction,
\cite{PS04} proposed using {\em $k$-dissimilarities}, 
$k\geq 3$, on $X$ rather
than 2-similarities as is the case when reconstructing
edge-weighted phylogenetic trees 
from distances (see also \cite[p.176]{fel-03}
and \cite{HHMS12,war-10,dez-ros-00} and the references 
therein for recent work on such objects which are sometimes
also called $k$-way similarities, $k$-way distances,
and $k$-semimetrics).  
It would be interesting to know what can be said about
lassoing and corraling of $X$-trees within this more general framework.\\

\noindent{\bf Acknowledgement}
A.\,-A.\,Popescu thanks the Norwich Research Park (NRP) for support. The authors thank the referees for their helpful comments.




\bibliographystyle{abbrvName} 
\bibliography{ultrametric}

\end{document}